\documentclass[a4paper]{article} 

\usepackage[activate={true,nocompatibility},final,tracking=true,kerning=true,spacing=true,factor=1100,stretch=10,shrink=10]{microtype}
\microtypecontext{spacing=nonfrench}
\usepackage{amsmath}%
\usepackage{amsfonts}%
\usepackage{amssymb}%
\usepackage{amsthm}
\usepackage{graphicx}
\usepackage{mathtools}
\usepackage[utf8]{inputenc}
\usepackage[english]{babel}
\usepackage{url,hyperref}
\usepackage{color}
\usepackage{varioref}
\usepackage{tcolorbox}
\usepackage{needspace}

\usepackage{caption}

\usepackage{multirow}

\usepackage{tikz}
\usetikzlibrary{matrix,arrows,decorations.pathmorphing}
\usepackage{tikz-cd}
\tikzset{commutative diagrams/.cd}

\usepackage[natbib, backend=biber, style=alphabetic, backref=true, maxbibnames=99]{biblatex}
\bibliography{bibliography}

\usepackage{pgf,tikz}
\usepackage{mathrsfs}
\usetikzlibrary{arrows}

\usepackage{epigraph}
\usepackage{listings}
\usepackage{mathrsfs}
\usepackage{dsfont}
\usepackage{enumerate}
\renewcommand\labelenumi{(\roman{enumi})}
\renewcommand\theenumi\labelenumi

\usepackage{float}

\usepackage{verbatim}

\usepackage{color,soul}

\usepackage[a4paper, total={5.9in, 8.5in}]{geometry} %was 5.5, 8

\usepackage{algorithm}
\usepackage{algpseudocode}

\newtheorem{theorem}{Theorem}%[subsection]
\newtheorem{lemma}[theorem]{Lemma}
\newtheorem{prop}[theorem]{Proposition}
\newtheorem{cor}[theorem]{Corollary}

\theoremstyle{definition}
\newtheorem{definition}[theorem]{Definition}

\theoremstyle{remark}
\newtheorem{remark}[theorem]{Remark}

\numberwithin{theorem}{section}
\numberwithin{equation}{section}

\usepackage[toc,page]{appendix}

\newcommand{\abs}[1]{\left\lvert#1\right\rvert}

\newcommand{\R}{{\mathbb{R}}}

\newcommand{\N}{{\mathbb{N}}}

\newcommand{\cA}{\mathcal{A}}
\newcommand{\cEvent}{\mathcal{E}}
\newcommand{\E}[1]{\mathbb{E}\left[ #1 \right]}
\newcommand{\Eover}[2]{\mathbb{E}_{#1}\left[#2\right]}

\newcommand{\Prob}[1]{\mathbb{P}\left( #1 \right)}

\usepackage{dsfont}
\newcommand{\ind}[1]{\mathds{1}_{#1}}
\newcommand{\given}{\ \middle|\ }

\DeclareMathOperator{\Vol}{vol}

\usepackage{bussproofs}

\newcommand{\vertiii}[1]{{\left\vert\kern-0.25ex\left\vert\kern-0.25ex\left\vert #1 
		\right\vert\kern-0.25ex\right\vert\kern-0.25ex\right\vert}}

\newcommand{\eps}{\varepsilon}
\newcommand{\q}{q^*}
\newcommand{\vol}{{\rm vol}}

\newcommand{\dv}{\mathbf{d}}
\newcommand{\fD}{\Delta^9}

\begin{document}
	
	\lstset{language=matlab}
	
	\definecolor{light-red}{rgb}{0.8,0,0}
	\definecolor{light-blue}{rgb}{0,0,0.8}
	\definecolor{pale-pink}{rgb}{1,0.95,0.95}
	\definecolor{cream-white}{rgb}{1,0.99,0.99}
	\definecolor{pure-white}{rgb}{1,1,1}
	\definecolor{pale-green}{rgb}{0.6,0.89,0.68}
	
	\title{Universal lower bound for community structure of sparse graphs}
	
	\author{Vilhelm Agdur, Nina Kam\v{c}ev and Fiona Skerman}
	\date{\today}

	\maketitle
         \begin{abstract}
            We prove new lower bounds on the modularity of graphs. Specifically, the modularity of a graph $G$ with average degree $\bar d$ is  $\Omega(\bar{d}^{-1/2})$, under some mild assumptions on the degree sequence of $G$. The lower bound $\Omega(\bar{d}^{-1/2})$ applies, for instance, to graphs with a power-law degree sequence or a near-regular degree sequence.

            It has been suggested that the relatively high modularity of the Erd\H{o}s-R\'enyi random graph~$G_{n,p}$ stems from the random fluctuations in its edge distribution, however our results imply high modularity for any graph with a degree sequence matching that typically found in~$G_{n,p}$. % from its degree sequence.
            
            The proof of the new lower bound relies on certain weight-balanced bisections with few cross-edges, which build on ideas of Alon [Combinatorics, Probability and Computing (1997)] and may be of independent interest.
        \end{abstract}

\section{Introduction}

In numerous real-world examples of graphs, we anticipate a certain community structure -- for instance, people form friend groups, neurons cluster into functional units, academic papers divide into subfields. To infer this structure from graph data, several metrics have been proposed to evaluate the quality of vertex partitions. One of the most widely used metrics is \emph{modularity},  introduced by Newman and Girvan~\cite{NewmanGirvan}.

Each vertex partition $\cA$ of a graph is given a  \textit{modularity score} $q_{\cA}(G)$, with higher scores taken to indicate that a partition better captures the community structure of a graph. In practice, for large networks, community detection is performed through algorithms that iteratively try to optimise this score~\cite{lf11}, such as the Louvain~\cite{louvain} or Leiden~\cite{traag2019leiden} algorithms.

The modularity score of a partition $\cA$ of a graph $G$ with $m$ edges is given by
\begin{equation}\label{eq:defmod}
q_\cA(G) = \sum_{A\in \cA} \frac{e(A)}{m} - \left(\frac{\Vol(A)}{2m}\right)^2\end{equation}
where the sum runs over parts $A \subseteq V(G)$ of the partition, $e(A)$ is the number of edges within part~$A$, and the volume $\Vol(A)$ is the sum of the degrees of the vertices in part $A$.

As can be seen, the formula for modularity consists of two terms balancing against one another -- one which partitions with edges within the parts, and a sum of squares term that rewards having many small parts, or for a fixed number of parts rewards parts of approximately equal volume. We call these terms the \emph{coverage} or \emph{edge contribution} and the \emph{degree tax} respectively.

The score of any partition is between $-0.5$ and $1$~\cite{nphard}. The modularity of a graph, $\q(G)$, is the maximum of $q_\cA(G)$ over all partitions $\cA$ of $G$. It is easy to see that the partition that puts all vertices in the same part gets a score of exactly zero, so that $\q(G)$ is always between zero and one.

While this gives us a reasonable way of comparing two different partitions of the same graph and telling which is better, it does not immediately give us a way of taking a graph and telling if it has a significant community structure or not. If you are given a graph, and compute that its modularity score is $0.23$, does this mean the graph has or does not have a community structure?

One might initially hope that random graphs with no underlying structure would have modularity essentially zero, but this turns out not to be true, at least if the graph is sparse, which many real-world graphs are. The binomial random graph, $G_{n,p}$, is likely to have high modularity so long as the average degree is bounded~\cite{GPA04,ERmod,reichardt2006statistical}. As proved by Ostroumova, Pra\l at and Raigorodskii~\cite[Theorem 6]{opr17}, any graph with maximum degree $o(n)$ and average degree $\bar d$ has modularity at least about~$2\bar{d}^{-1}$. For random graphs with a given bounded degree sequence (unders some natural assumptions), this lower bound was improved to $(2+\eps)\bar{d}^{-1}$ by Lichev and Mitsche~\cite{lichev_mitsche_2022}.

We give another result in this direction, showing that any graph whose degree sequence is not too heavy-tailed will have modularity $\Omega\left(\bar{d}^{-\frac{1}{2}}\right)$. One motivation for this result are applications to graphs with a \textit{power-law} degree sequence, and specifically, to preferential-attachment graphs, which are discussed in Section~\ref{subsec:powerlaw} extending a result of~\cite{opr17}. The following statement is a concise version of the result as $n \to \infty$ (so $o(1)$-terms are with respect to $n$), whereas the error terms are stated explicitly as Proposition~\ref{prop:main}. 
\begin{theorem}\label{t:main}
Let $G$ be an $n$-vertex graph with average degree $\bar d \geq 1$, $L = \left\{v \in G\given d_v < C \bar{d}\right\}$ for some $C > 1$, and assume that  $\Vol(L) \geq (1 + \gamma) m = (1 + \gamma)\frac{n \bar d}{2}$ for some $\gamma >0$.  

If ${\Delta(G)} n^{-1}= o(1)$ and $\bar d ^{10}n^{-1}=o(1)$, then $$\q(G)\geq \frac{0.26 \gamma}{\sqrt{C \bar{d}}}(1+o(1)).$$
\end{theorem}

One way to interpret the result is that, assuming $\bar{d}=o(n^{1/10})$, the only obstruction to modularity~$\Omega(\bar{d}^{-1/2})$ is if we have a minority of vertices which contain at least half the volume of the graph. This happens for unbalanced bipartite graphs, as discussed below.

The $\Omega\left( \bar{d}^{-1/2} \right)$ is the best lower bound we could hope for without imposing more conditions, because there exist families of graphs which achieve this bound. For example, $d$-regular graphs~$G_{n,d}$, for large enough $d$, have modularity $\q(G_{n,d})=\Theta\left(d^{-1/2}\right)$~\cite{treelike}. Another example are Erd\H{o}s-Reny\' i random graphs, see the following section, as well as the Chung-Lu model, see Section~\ref{sec:upper-random}.

\paragraph{Modularity from fluctuations in random graphs? Or automatically by average degree?}
Guimer\`a, Sales-Pardo and Amaral published a highly influential paper showing the binomial random graph $G_{n,p}$ can have high modularity~\cite{GPA04}. They estimated it to have modularity $\Theta\left((np)^{-1/3}\right)$, meaning the modularity does not go to zero for constant average degree, the usual regime of interest for real networks. Using deep but non-rigorous insights from statistical physics Reichardt and Bornholdt~\cite{reichardt2006statistical} conjectured the modularity of $G_{n,p}$ to be $\Theta\left((np)^{-1/2}\right)$ whp, and this was confirmed to hold whp for $1/n \leq p \leq 0.99$ 
in \cite{ERmod}. 
	
Notice that this matches the bound in Theorem~\ref{t:main}. To be precise since the average degree of~$G_{n,p}$ is tightly concentrated about $(n-1)p=np(1+o(1))$, \cite{ERmod} implies that for $1/n \leq p \leq 0.99$ and for $G\sim G(n,p)$ whp we have $\q(G)=\Theta(\bar{d}(G)^{-1/2} )$.

Thus, our result shows that these lower bounds on the modularity of $G_{n,p}$ hold simply because of its average degree and the well-behaved nature of its degree sequence, without needing any appeal to fluctuations or any other particular feature of the model -- the same bound holds for any graph with a similar degree sequence.

Furthermore, our lower bounds offer a certain level of validation to the concept of modularity as a measure for community structure. It would be less than satisfactory if there existed graphs with considerably lower modularity than  a random graph with a \textit{similar} degree sequence. Our results, on the other hand, imply that random graphs do, in a sense, have the minimum achievable modularity.

\paragraph{Modularity results}
The modularity of graphs from random graph models and the relationship between graph properties and modularity has received much recent attention. We have mentioned already the deterministic lower bound of $2\bar{d}^{-1}$ for graphs with sublinear maximum degree~\cite{opr17} and results for random regular~\cite{treelike} and Erd\H{o}s-R\'enyi graphs~\cite{ERmod}. For random cubic graphs Lichev and Mitsche~\cite{lichev_mitsche_2022} proved the modularity whp lies in the interval $[0.667, 0.79]$ and more generally that random graphs with a given  degree sequence have modularity at least $(2+\eps)\bar{d}^{-1}$. Preferential attachment graphs with $h\geq 2$ edges added at each step whp have modularity at most~$15/16$ and at least $\Omega(\bar{d}^{-1/2})$~\cite{opr17}, see also the next section. Sampling a random subgraph of a given graph by including each edge independently with probability $p$ whp has modularity approximating that of the underlying graph for $p$ such that the expected degree in the sampled graph is at least a large constant~\cite{mcdiarmid2021modularity}.

There are also graphs known to be `maximally modular'~\cite{modgraphclasses}, i.e.\ with modularity tending to 1 as the number of edges tends to infinity. It has been shown that graphs with sublinear maximum degree from a minor-closed class~\cite{lasonsulkowska2023modularity} and (whp) hyperbolic graphs~\cite{hyperbolic} and spatial preferential attachment~\cite{opr17} are maximally modular. In another direction, see~\cite{gosgens2023hyperspherical} for a geometric interpretation of modularity.

\paragraph{Tightness of Theorem~\ref{t:main} - complete bipartite and random graphs.}
The result is tight in two senses. Firstly, the $\gamma>0$ condition is necessary and secondly the $\Omega(\bar{d}^{-1/2})$ is the best lower bound we could hope for without imposing more conditions.

To see that $\gamma>0$ is necessary, we note that for any even $d$ we may construct a graph $G$ with average degree about $d$ with $\q(G)=0$ and such that $\Vol(L)=\Vol(G)/2$. It is known that complete bipartite graphs have modularity zero~\cite{bolla2015spectral,majstorovic2014note}.
Take the complete bipartite graph $G=K_{d/2, t }$, and note the average degree is $\bar{d}=dt/(d+t)$ and thus for sufficiently large $t$ we have $\bar{d}\approx d$. The graph has many vertices with degree $d/2$, few vertices of degree $t$ and both sets have volume $\Vol(G)/2$. Thus the set $L$ in the theorem statement will consist of vertices of degree $d/2$ and have volume $\Vol(G)/2$, yielding $\gamma=0$. 

To see that $\Omega(\bar{d}^{-1/2})$ is the best lower bound possible without imposing more conditions, we recall that a random $d$-regular graph has modularity at most $2d^{-1/2}$~\cite{treelike}. Moreover, Corollary~\ref{cor.ubCOLanka} gives examples of graphs with a large family of possible degree sequences and modularity~${\Omega(\bar d^{-1/2})}$.

\subsection{Application to power-law graphs} 
\label{subsec:powerlaw}

In this section, we discuss two applications of Theorem~\ref{t:main}. Informally, graphs with a \textit{power-law degree} sequence and preferential-attachment graphs have modularity $\Omega(\bar d^{-1/2})$, generalising a result of~\cite{opr17}.

Many real-world graphs follow a power-law degree distribution, for instance the World Wide Web, genetic networks and collaboration networks \cite{ab99,barabasi-book,cl04}. This means that the proportion of vertices of degree $k$ is $O(k^{-\tau})$ for a parameter $\tau$, called the \textit{shape coefficient}.  
Most examples found in literature have the \textit{shape coefficient} $\tau$ in the interval $(2, 3]$ -- for example roughly~$2.2$ for the Internet or~$2.3$ for the movie actors network~\cite[Chapter 8]{vanderhofstad17}. For $\tau >2$, that is, as soon as the first moment of the sequence is well-defined, most of the volume is on the vertices whose degree is near average. This allows us to apply Theorem~\ref{t:main} and obtain the following lower bound.

%%%%%%%%%%%%%%%%%%%%%%

\begin{theorem}\label{t:powerlaw}
    Let $G$ be a graph with degree sequence $\dv = (d_i)_{i \in [n]}$, with average degree $\bar d$, satisfying 
    \begin{equation}\label{eq:powerlaw-hyp}
        \frac 1n |\{i: d_i \geq k \}| \leq A \bar d^{\tau -1} k^{1-\tau}
    \end{equation}
    for all $i$, with constants $\tau >2$ and $A >0$. For $b = 0.1 \left( \frac{(\tau-2)}{8A} \right)^{\frac {1}{2(\tau -2)}}$ and sufficiently large $n$,
    $$\q(G) \geq b \bar d^{-1/2}.$$
\end{theorem}

As mentioned, the best previously known lower bound for the modularity of graphs satisfying~\ref{eq:powerlaw-hyp} is $2\bar d^{-1}$~\cite{opr17} - since their max degree is sublinear. The modularity was already known to be~$\Omega(\bar{d}^{-1/2})$ for preferential attachment graphs (with $\delta=0$)~\cite{opr17} so our Theorems~\ref{t:powerlaw} and~\ref{thm:pa} generalise this. 
 
\paragraph{Modularity of random power-law graph models}
	
There are numerous random graph models which aim to model existing networks with a power-law degree distribution (often referred to as \textit{scale-free networks}). They fall into two basic categories,
\begin{enumerate}
   \item graphs whose degree sequence is specified a priori, and
    \item graphs in which the degrees emerge from stochastic local growth rules, such as preferential-attachment graphs.
\end{enumerate}

For (i), a lower bound on the modularity follows directly from Theorem~\ref{t:main}, assuming that the empirical degree sequence is close to the prescribed one. Notice that this holds by definition for random graphs with a given \textit{fixed} degree sequence, so Theorem~\ref{t:main} trivially applies to the uniform model, extending the results of~\cite{treelike} for $G_{n,d}$. It also implies a lower bound for the Chung-Lu model for graphs with a given \textit{expected} degree sequence. For the Chung-Lu model, we are also able to give a matching (up to constant factor) lower bound on the modularity, see Section~\ref{sec:upper-random}.

The models from category~(ii) are usually based on the preferential attachment model (PAM) \cite{ab99,brst01}, which is described in more detail in Section~\ref{sec:pam}. The preferential attachment model was introduced in a seminal paper by Albert and Barab\'asi~\cite{ab99}, which also demonstrated its ability to explain the emergence of \textit{scale-free} networks and laid the foundation for the study of complex networks. For more information on mathematical properties and applications of the PAM, see for instance~\cite{barabasi-book,vanderhofstad17,cl02}.
For PAM-type models (ii), it is not easy to prove rigorous results about the degree sequence, and controlling high-degree vertices seems particularly inaccessible (see, e.g.,~\cite{cl04}[Section 2.2] and Proposition~\ref{prop:pam-deg}~\ref{pa-degseq} in this paper). For this reason, Theorem~\ref{t:powerlaw} does not apply directly, but we demonstrate how Theorem~\ref{t:main} can be applied to the class of preferential attachment models presented in Section~\ref{sec:power_law}.

\subsection{Key Techniques}
    \paragraph{Alon's bisection method}
    Throughout this section, $G$ is a given graph with average degree $\bar d$ and we wish to find a bisection of $G$ with modularity $\Omega(\bar d^{-1/2})$. Central to our proof is the method of Alon~\cite{ab99} which gives a bisection of the graph with $n (\bar d/4 - \Omega(\bar d ^{1/2}))$ edges between the two parts. Notice that the second term $\Omega(\bar d^{1/2})$ is the deviation from a random bisection. A crucial idea in this method is to find a pairing of the vertices which does not \textit{interfere} with the edges of the graph in undesired ways, so that a randomised bisection of the vertices along those pairs can be analysed (see Lemma~\ref{blackboxbound_eAB}).

    In obtaining bisections with \textit{high} modularity, we face two obstacles -- the degree tax, and the fact that Alon's bisection technique only applies to graphs with maximum degree $O(n^{1/9})$.

    \paragraph{Pairings which equalise the volume} Regarding the degree-tax obstruction, by definition~\eqref{eq:defmod}, if a bisection obtained above is to have high modularity, it needs to have degree tax as small as possible, i.e.~the two parts need to have approximately the same volume to give degree tax near a half (see definition p\pageref{eq:defmod}). 
    Thus our problem is to find a bisection of $G$ with the same guarantee of few edges between parts, but also such that the volume of the two parts is similar. The technical result allowing us to partition a graph while controlling the volume is Lemma~\ref{matchinglemma} where we find a pairing of almost all vertices such that the vertices of each pair are \textit{near} in the degree-ordering of the vertices, but the pairing is still suitable for Alon's bisection technique to apply. This together with a load-balancing result Lemma~\ref{lem.obscure} yields Theorem~\ref{theorem_withoutcutoffs} -- informally, high modularity for graphs with maximum degree $o(n^{1/9})$.

    \paragraph{Processing high-degree vertices} However, the constraint that $\Delta(G) = o(n^{1/9})$ is still too strong for many desired applications -- for instance, graphs with a power-law degree sequence often have a significantly higher maximum degree (see Section~\ref{sec:power_law}). To circumvent this problem, we essentially apply the bisection method as above to the \textit{bulk} of the graph, that is, to the vertices whose degrees are not too far above the mean which we denote by $L$. Then we randomly divide $H=V(G)\setminus L$, the high-degree vertices, into our two parts. With positive probability, such a partition will have modularity $\Omega(\bar d^{-1/2})$ -- the main contribution to positive modularity will come from partitioning $L$, and for $H$, we only need to show that they behave approximately \textit{as expected}, even with respect to the previously found partition of $L$.

\section{Weight-balanced bisections}\label{sec:weight_balanced_bisections}
We now describe the bisection idea due to Alon~\cite{alon_1997}. The method starts from a convenient matching on $V(G)$, which we now define.
\begin{definition}
		Given a graph $G = ([n], E(G))$ and a matching $M$ disjoint from $E(G)$, a \textit{short loop} of $G$ and $M$ is a loop of length at most twelve containing between one and three edges from $M$ and never more than three consecutive edges of $G$.
\end{definition}
Note that in particular, this definition implies that $M$ and $G$ are edge-disjoint. Given such a matching $M$, Alon proposed and analysed a simple randomised algorithm which splits the vertices of the graph $G$ `along' $M$. We only describe the idea informally as we will not explicitly use it in this paper, it will suffice to use the result Theorem~\ref{alonblackbox} as a black-box. The first step is to orient the edges of $M$ independently and uniformly at random, which splits the vertex set into the set of sources and sinks in this orientation. An edge $uv$ of $M$ is marked \textit{active} if reorienting $uv$ would not increase the number of `cross-neighbours' of both $u$ and $v$ in the opposite part. The second step is to uniformly resample the orientations of the active edges, and to output the induced partition.

This partition is shown to have \textit{very few} cross-edges with positive probability, and the requirement for no short loops is important in the analysis. Below we state the result in a self-contained form. In~\cite{alon_1997}, the computations are carried out for $d$-regular graphs, but the argument covers arbitrary degree sequences verbatim, and this is also stated in the concluding remarks in~\cite{alon_1997}. Throughout the paper, $c = \frac{3}{8\sqrt 2} \approx 0.265$ is a fixed constant.
\begin{theorem}[\cite{alon_1997}]\label{alonblackbox}
		Given any graph $G$, and any perfect matching $M$ on $[n]$ disjoint from $E(G)$ such that there exist no short loops of $G$ and $M$, there exists a $U \subset [n]$ such that $M \subset U \times U^c$, and
		\begin{equation}
			e_G(U, U^c) \leq  \frac{1}{2} \sum_{i=1}^n d_i \left(\frac{1}{2} - \frac{c}{\sqrt{d_i}}\right).
		\end{equation}
		%where $c \in \left(0,\frac{1}{2}\right)$ is some absolute constant.
	\end{theorem}
It is convenient to identify the vertex set of our graphs with $[n]$, since this gives a natural ordering on the vertices of $G$. Later, we will choose a specific vertex ordering to which the following lemma will be applied.
	\begin{lemma}\label{matchinglemma}
		Given any graph $G = ([n], E(G))$ with maximum degree $\Delta>1$, there exists a partial matching $M$ on $[n]$ disjoint from $E(G)$ such that the following holds:
		\begin{enumerate}
			\item[1.] \label{it:bandwidth} For any $vw \in M$, $\abs{v-w} \leq \fD$.
			\item[2.] \label{it:cyclesinG} There are no short loops of $G$ and $M$.
			\item[3.] \label{it:large} $[n-\fD] \subset V(M)$
		\end{enumerate}
	\end{lemma}
	Note that the last statement in particular means the lemma is void when $n < \fD$.
	\begin{proof}
		Let $H = G^{3}$, the graph where there is an edge $(u,v)$ if there is a path of length at most~$3$ from $u$ to $v$ in~$G$. It is straightforward to verify that our condition~2 on the matching $M$ is implied by
		\begin{enumerate}
			\item[2'.] \label{it:cycles} There does not exist any cycle of length two, four or six consisting of alternating edges from $H$ and $M$. (In particular, $H$ and $M$ are edge-disjoint.)
		\end{enumerate}
        Recalling that the maximum degree of $G$ is $\Delta>1$, we have that the maximum degree of $H$ is 
        $$\Delta(H) \leq \Delta + \Delta(\Delta-1) + \Delta(\Delta-1)^2 \leq \Delta^3 - 1.$$
		
		Intuitively, the idea of the proof is to construct the matching greedily, taking the smallest currently unmatched vertex $v$ and joining it to the first available vertex. It will be enough to show that until the very last rounds, the number of unavailable vertices will be not too large. Vertices are made unavailable when they are already incident to edges in the matching or when there is a particular \textit{dangerous} configuration of alternating edges (to ensure we do not violate property~2'). Loosely speaking, we maintain an upper bound on the number of unavailable vertices using property~1, which guarantees that the matched vertices are not too far from one another, and is established at the start of each step.

		We construct this matching $M$ using the following greedy algorithm, see Figure~\ref{fig:phases}. We identify the graphs $H$ and $M$ with their edge sets. For a matching $M$, we write $V(M)$ for the set of vertices incident to an edge of $M$.
		\begin{algorithm}
			\caption{Construction of $M$}\label{alg:constr_M}
			\begin{algorithmic}
				\State $M_1 \gets \emptyset$
				\For{$v = 1,2,\ldots,n - \fD -1$}
				\If{$v \not\in V(M_v)$}
				\State Let $F^+_v$ be the set of all $u \in \{v+1,\ldots,n\}$ such that there is a path between $u$ and $v$ consisting of alternating edges from $H$ and $M_v$  of the form $H$, $H M_v H$ or $H M_v H M_v H$. 
				\State Pick the least $w \in [n]$ such that $w \not\in F^+_v \cup V(M_v)$.
				\State $M_{v+1} \gets M_v \cup vw$
				\Else
				\State $M_{v+1} \gets M_v$
				\EndIf
				\EndFor
			\end{algorithmic}
		\end{algorithm}
		
		We will first show that the argument terminates, i.e.~that a suitable vertex $w$ can be found as long as $v \leq n - \fD$, and then that the resulting matching $M_{n-\fD + 1}$ has the desired properties.\\
		
		\noindent\emph{Claim: } Let $M_v^+=V(M_v) \cap \{v+1, \ldots, n \}$. If $v \notin V(M_v)$ then $|F_v^+\cup M_v^+| \leq \fD-1$.
		
		To begin the proof of the claim, similarly to $F_v^+$, define $F_v^-$ to be the set of $x \in \{1, \ldots v-1\}$ such that there is a path between $x$ and $v$ %consisting of alternating edges from $E$ and $M_v$ 
		of the form $H$, $HM_vH$ or $HM_vHM_vH$ and define $F_v=F_v^+ \cup F_v^-$.
		
		Let $u \in M_v^+$, and note we have $uw \in M_v$ for some $w<v$. But given the greedy algorithm to construct $M$, and $v<u$, this implies that $v$ was not a valid choice for $w$ to pick, that is, it must be the case that $v \in F_w^+ \cup V(M_w)$. Since $M_w \subseteq M_v$ and $v \notin V(M_v)$, this implies $v \in F_w^+$. Thus there is a path between $w$ and $v$ of the form $H$, $HM_wH$ or $HM_wHM_wH$ and so $w\in F_v^-$, again using $M_w \subseteq M_v$. In particular, we have shown that for each $u \in M_v^+$ there is a distinct $w \in F_v^-$ and hence \begin{equation}\label{eq:fwdsandback}
			|F_v^+\cup M_v^+|\leq|F_v|.
		\end{equation}
		
		The number of vertices $u$ incident to $v$ is at most $\Delta(H) $. Similarly, the number of paths starting from $v$ of the form $H M_v H$ is at most $\Delta(H)^2$ and of the form $H M_v H M_v H$ is at most $\Delta(H)^3$. Thus $|F_v|\leq \Delta(H)+\Delta(H)^2+\Delta^3 \leq (\Delta(H)+1)^3 -1 \leq \fD -1$ and by \eqref{eq:fwdsandback} we have shown the claim.
		
\begin{figure}
    \includegraphics[width=\textwidth]{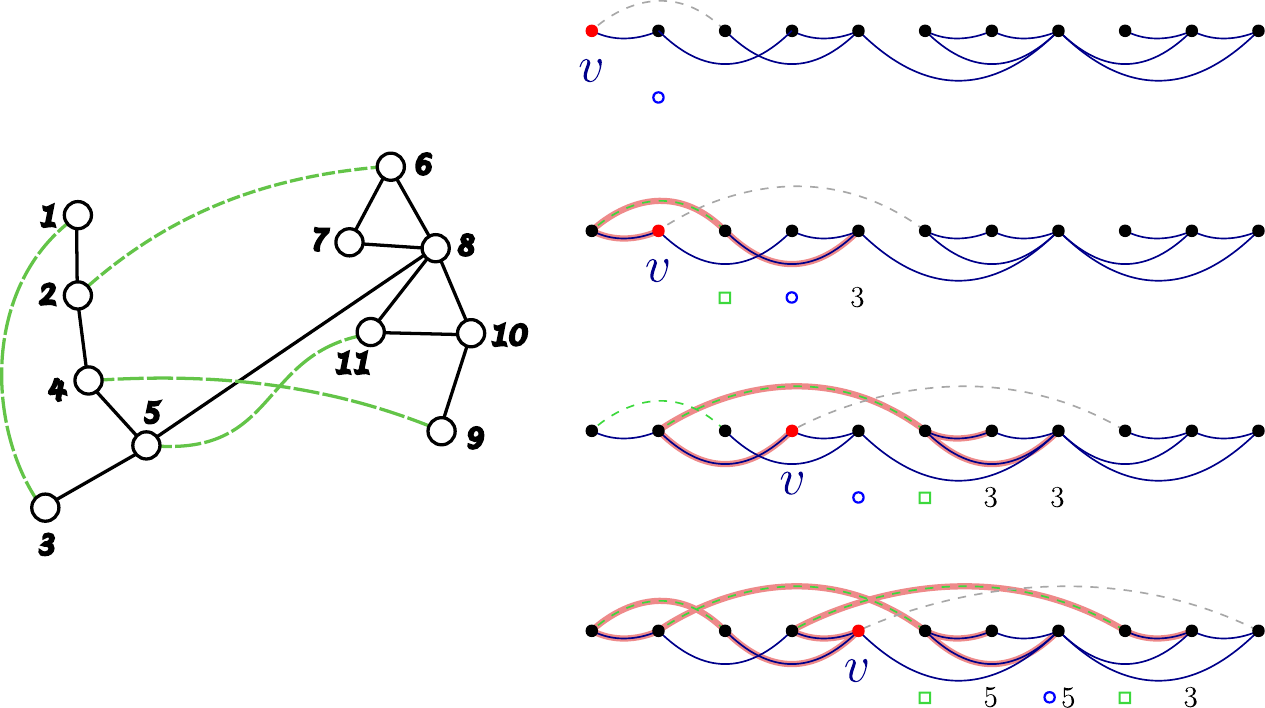}
    \centering
    \caption{\footnotesize An illustration of the process of constructing a matching with our greedy algorithm. \newline
    The graph $H$  is on the left, with the edges of the resulting matching $M$ in dashed green lines. To the right, we see each of the four time-steps in which a new edge is added, with indications of why the chosen new edge (in grey) is the one being added. In the first time-step, we cannot choose to add edge $12$ to our matching, because there already is an edge of $H$ between $1$ and $2$ -- we illustrate this with a blue circle ({$\color{blue}\circ$}). \newline
    In the second time-step, we cannot add edge $23$, because $3$ is already in the matching (indicated by a green square,~{\color{green}$\square$}), we cannot add edge $24$ because that is already an edge in $H$, and we cannot add edge $25$, because that would add a path of length three of $H$ and $M$ -- we illustrate this by highlighting the path, and a $3$ underneath vertex $5$.\newline
    The construction proceeds similarly for two more steps (vertex three is skipped because it is already in the matching when we get to it) -- in the last step, two vertices are forbidden by length-five paths -- before we finally arrive at vertex seven, for which the algorithm cannot find a match, and it terminates.}

    \label{fig:phases}
\end{figure}
		
		The claim above implies that $\{v+1, \ldots, v + \fD \} \setminus (V(M_v) \cup F_v) \neq \emptyset$ for $v \leq n- \fD$, so indeed, there is a valid choice for $w$ in each step, and the algorithm terminates. Let $M= M_{n- \fD + 1}$. For each vertex $v$, $V(M_{v+1})$ contains the initial segment $[v]$, so in particular, $M$ contains $[n- \fD]$, certifying property~3.%~\ref{it:large}. 

      Finally, let us show that $M_v$ satisfies condition~$2'$ for each $v$. This clearly holds for $M_1 =\emptyset$, and suppose that some $M_v$ satisfies condition~$2'$. If $M_{v+1} = M_v$, $M_{v+1}$ clearly satisfies condition~$2'$. Moreover, if $M_{v+1} = M_v \cup \{vw\}$, then $vw$ does not close an alternating cycle of length 2, 4 or~6 because $w$ does not lie in $F_v^+$. In either case, $M_{v+1}$ satisfies condition~$2'$, as required.
        
      In particular, $M_{n- \Delta^9+1}$ satisfies condition 2', which completes the proof.
	\end{proof}
	
	We shall use the following load-balancing result to show that the two parts of our partition have similar volume - see \cite{obscurepaper}[Lemma~2.2] or the thesis \cite{skerman2016modularity}[Lemma~2.1.3].
	\begin{lemma}\label{lem.obscure}
		Suppose that $f: [n] \to \R$ is some non-increasing function. Then, for any perfect matching $M$ on $[n]$  such that for every $(i,j) \in M$, $\abs{i-j} \leq L$, and for any orientation of the edges of $M$, it holds that
		$$\bigg| \sum_{(i,j) \in M} f(i) - f(j) \bigg| \leq L|f(n) - f(1)|.$$
	\end{lemma}
	We can assemble these lemmata into the following proposition.
	\begin{prop}\label{partitionproposition}
		Let $G = ([n], E)$ be a graph with maximum degree $\Delta$ satisfying $\Delta^9 \in \left[1,\frac{n}{2}\right)$, and let $w_{\max} = w_1 \geq \ldots \geq w_n=w_{\min}$ be non-negative vertex weights. Let $\bar{w}$ be the average vertex weight, $\bar{w} = \frac{1}{n}\sum_{u} w_u$, and for a set $A$ of vertices, let $w(A) = \sum_{u\in A} w_u$.
  
        There exists a partition $\{A, B, R\}$ of $V(G)$ such that
		\begin{enumerate}
			\item $\abs{A} = \abs{B}$,
			\item $R \subset \{n - \Delta^9 + 1, \ldots, n\}$,
			\item $e(A,B) \leq  \frac{1}{2} \sum_{v \in A \cup B}^n d_v \left(\frac{1}{2} - \frac{c}{\sqrt{d_v}}\right)$ ,
			\item $\abs{w(A) - w(B)} \leq \Delta^9(w_{\max} - w_{\min})$ and
			\item $\max_{v \in R} w_v \leq 2\bar{w}$.
		\end{enumerate}
	\end{prop}
	\begin{proof}
		We may apply Lemma~\ref{matchinglemma} to our graph -- recall that we have assumed our weights $w_i$ are decreasing, which will imply that the unmatched vertices will be the ones of lowest weight. We obtain a matching $M$ consisting only of edges $ij$ with $\abs{i - j} < \Delta^9$. Let $R$ be the set of vertices not matched by $M$. Lemma \ref{matchinglemma} also tells us that $\abs{R} \leq \Delta^9$, and in fact, $R$ is contained in the final segment $\{n - \Delta^9 + 1, \ldots, n\}$.
		
		Let $G' = G \setminus R$. The graph $G'$ along with the matching $M$ fulfils the conditions of Theorem~\ref{alonblackbox}, by the construction of $M$. Hence, we obtain a set $U \subseteq [n]$ such that $M \subseteq U \times U^c$, and
		\begin{align*}
			e_{G'}(U, U^c) &\leq  \frac{1}{2} \sum_{v \in G'} d^{G'}_v \left(\frac{1}{2} - \frac{c}{\sqrt{d^{G'}_v}}\right)
			\leq \frac{1}{2} \sum_{v\in G'}^n d_v \left(\frac{1}{2} - \frac{c}{\sqrt{d_v}}\right),
		\end{align*}
		where the second inequality follows from $c < \frac{1}{2}$.
		
		Since $\abs{i - j} < \Delta^9$ for all edges $ij$ of $M$, we can appeal to Lemma~\ref{lem.obscure} and get a bound on the difference of the weights of the sets, namely $\abs{w(A) - w(B)} \leq \Delta^9(w_{\max} - w_{\min})$ as desired.
		
		Finally, let us bound the weights of the unmatched vertices. We established that the remainder $R$ will be among the $\Delta^9$ vertices of lowest weight. Suppose for contradiction that $\max_{v\in R} w_v$ is larger than $2\bar{w}$ -- then, by the ordering of the vertices, \emph{all} the vertices $1, \ldots, n-\Delta^9$ must have weight at least $2\bar{w}$. So we can compute
		\begin{align*}
			\bar{w} = \frac{1}{n}\sum_{i=1}^n w_i \geq \frac{1}{n}\sum_{i=1}^{n-\Delta^9} w_i
			\geq \frac{n - \Delta^9}{n}2\bar{w} = \left(1 - \frac{\Delta^9}{n}\right)2\bar{w},
		\end{align*}
		and the fact that $1 - \frac{\Delta^9}{n} > \frac{1}{2}$ follows from our assumption that $\Delta^9 < \frac{n}{2}$, giving us the desired contradiction.
	\end{proof}

\section{From weight-balanced partitions to modularity bounds}\label{sec:without_cutoffs}
We now pivot back to considering modularity in particular, as our objective measure on partitions. %and hence we will be interested in constructing partitions which are nearly bisections where the two parts have balanced volumes. 
We start by showing a weaker version of our main theorem, that only applies under the assumption of a maximum degree bound, in order to illustrate the proof in a simpler setting. Then, to prove the main theorem, the main idea is essentially to apply Theorem~\ref{theorem_withoutcutoffs} to the bulk of the graph, that is, to the vertices whose degrees are not too far from the mean (and so we have a bound on the max degree in this bulk), and then randomly divide the high-degree vertices of the graph into our two parts.
	
Thus the main term of our main theorem is essentially the same as the main term of Theorem~\ref{theorem_withoutcutoffs}, because this bulk is where the main terms is gained; for the high degree vertices, we merely take a partition that yields roughly the expected number of cross-edges and does not interfere with the previous partition. 

 For technical reasons, it makes more sense to give a standalone proof of our main theorem that does not directly appeal to the following weaker theorem. We nevertheless include this theorem because its proof highlights the ideas at play without as many details to obscure them.

\subsection{Modularity bounds - an easy application of weight-balancing.}

The following theorem will follow quickly from our weight-balancing result, Proposition~\ref{partitionproposition}.

\begin{theorem}\label{theorem_withoutcutoffs}
For any graph $G$ such that $\Delta^9 \in \left[1,\frac{n}{6}\right)$, we have
    $$\q(G) \geq \frac{c}{n}\sum_{i=1}^n \frac{\sqrt{d_i}}{\Bar{d}} - \frac{\Delta^{20}}{2(n\bar d)^2} 
    .$$
 \end{theorem}

To prove the theorem we will use Proposition~\ref{partitionproposition}, by taking the vertex degrees as the weights, which gives us a volume balanced nearly-bisection: two large sets $A$ and $B$ with similar volumes and a small remainder set $R$. The modularity score of the partition into these three sets will be high if the number of edges between $A$ and $B$ is significantly less than half the edges of the graph, the volumes of $A$ and $B$ are sufficiently similar and $R$ has a sufficiently small volume. The following lemma makes this precise. 

\begin{lemma}\label{lem.nearlybisection} Let $G$ be a graph and $\cA=\{A,B,R\}$ a vertex partition of $G$ with $\vol(R)\leq \vol(G)/3$. Then
    \begin{eqnarray*}
        q_\cA(G) &\!\geq \!& 
                     \frac{1}{2}  - \frac{e(A,B)}{e(G)} - \frac{(\vol(A)-\vol(B))^2}{2\vol(G)^2}
    \end{eqnarray*}
\end{lemma}

The proof of Lemma~\ref{lem.nearlybisection} is straightforward and so we defer the details to the appendix, see page~\pageref{lem.nearlybisectionagain}.

\begin{proof}[Proof of Theorem~\ref{theorem_withoutcutoffs}]
Since $G$ satisfies $\Delta^9 \in \left[1,\frac{n}{6}\right)$, it follows from Proposition~\ref{partitionproposition}, taking our vertex weights to be the degrees of the vertices, that there exists a partition $\{A, B, R\}$ of the vertices of $G$ with $\abs{A} = \abs{B}$ and such that
\begin{equation}\label{blackboxbound_eAB}
    e_G(A, B) \leq  \frac{1}{2} \sum_{i=1}^n d_i \left(\frac{1}{2} - \frac{c}{\sqrt{d_i}}\right),
\end{equation}
\begin{equation}\label{obscurelemmabound}
    \abs{\Vol_G(A) - \Vol_G(B)} \leq \Delta^9(\Delta - \delta) \leq \Delta^{10}
\end{equation}
and the remainder $R$ satisfies $\abs{R} \leq \Delta^9$ and $\max_{v\in R} d_v \leq 2\bar{d}$.
			
We now prove a lower bound on the modularity score of the partition $\{A, B, R\}$ and hence a lower bound on the modularity value of $G$. Recalling that $\sum_i d_i = 2e(G)$ and $e(G)=n\bar{d}/2$ the bound \eqref{blackboxbound_eAB} gives
\begin{equation}\label{blackboxbound_eAB_massaged}
    e_G(A, B) \leq  \frac{e(G)}{2} - \frac{c}{2}\sum_{i=1}^n \sqrt{d_i} = e(G)\bigg( \frac{1}{2} - \frac{c}{n}\sum_{i=1}^n \frac{\sqrt{d_i}}{\bar{d}}\bigg).
\end{equation}
To bound the volume of $R$ 
\[ \vol(R) \leq |R| \max_{v \in R} d_v \leq 2 \Delta^9 \bar{d}, \]
and thus for $\Delta^9 \leq n/6$, we have $\vol(R)\leq \vol(G)/3$ and so can apply~Lemma~\ref{lem.nearlybisection}. Now substituting the bounds in \eqref{obscurelemmabound} and \eqref{blackboxbound_eAB_massaged} into Lemma~\ref{lem.nearlybisection} gives the desired result.\end{proof}

\subsection{Modularity bounds without the max degree condition}
\label{sec:with_cutoffs}
In this section we prove our main theorem. The idea here is that we can apply the same method as we did for Theorem \ref{theorem_withoutcutoffs} for the bulk of the graph, and then deal with the high-degree vertices separately. By doing so, we can remove the condition on the maximum degree, instead replacing it with a mild condition on the upper tail.

The proof still uses the weight-balanced bisection with few edges across the parts to gain its main term, however we will now apply it just to a subgraph $G[L]$, where $L$ is the set of vertices whose degree is at most a constant multiple of the average degree. 

We then assign the vertices in $[n]\setminus L$ randomly to the two parts of our partition, using one method for vertices with degree at most $\sqrt{n}$ and one for the rest. Unlike in the bulk, where the weight-balanced bisection actually gains us our main term, in this part we can only hope to keep the additional error terms small, since our random assignments do not really use the structure of the graph.

Note the following simple expression for two-part modularity (which follows for example from Lemma~\ref{lem.nearlybisection} by taking $R$ to be the empty set). 
\begin{remark}\label{remark_twopartscomputation}
		For any graph $G$ and partition $\{A, B\}$ of its vertices, we have
		$$q(G, \{A, B\}) = \frac{1}{2} - \frac{e(A,B)}{m} - \left(\Vol(A) - \Vol(B)\right)^2\frac{1}{8m^2}$$
\end{remark}
We are now ready to prove Theorem~\ref{t:main}, which we restate here as a proposition with explicit error terms.
\begin{prop}
    \label{prop:main}
    Let $G$ be an $n$-vertex graph with average degree $\bar d \geq 1$, $L = \left\{v \in G\given d_v < C \bar{d}\right\}$ for some $C > 1$, and assume that  $\Vol(L) \geq (1 + \gamma) m = (1 + \gamma)\frac{n \bar d}{2}$ for some $\gamma >0$.  If $\vartheta = (C \bar d)^{10}n^{-1} < \frac{1}{2}\left(1 - \frac{1}{C}\right)$, then
\begin{align*}
    \q(G) &\geq \frac{0.26}{\sqrt{C \bar{d}}} \left(\gamma - \frac{2 \vartheta}{C \bar d}  \right) - \frac{\vartheta^2}{2 \bar d^2} - \frac{3}{8\sqrt{n}}-\frac{4\Delta(G)^2}{n^2 \bar d^2}.
\end{align*}
Moreover, if $\frac{\Delta(G)}{n} = o(1)$ and $\vartheta=o(1)$, then $\q(G)\geq \frac{0.26 \gamma}{\sqrt{C \bar{d}}}(1+o(1))$.
    \end{prop}

\begin{proof}%[Proof of Theorem~\ref{t:main}]
 To prove the bound, we randomly construct a bipartition with expected modularity score at least as claimed, and thus conclude that there exists a bipartition achieving at least that score. As in Theorem~\ref{theorem_withoutcutoffs}, we use the weight-balancing result, Proposition~\ref{partitionproposition}, this time applying it to just the low-degree vertices, $L$, to get a partition into $\{A, B, R\}$. For the random partitioning step, we take the remainder $U=L \cup R$ and randomly divide it into two parts $U_A$ and $U_B$. (Here, we break into vertices $U^+$ with degree at least $n^{1/2}$ and the remainder $U^-$ and have slightly different procedures for $U^+$ and $U^-$.)

Let $G' = G[L]$ and $S = V(G)\setminus L$. Given a vertex $v \in L$, let $d_v'$ be the degree of $v$ in $G'$, that is, its number of neighbours in { $L$}. 
We will apply Proposition \ref{partitionproposition} to the graph $G'$, using the degrees $d_v'$ as our vertex weights. This will require us to bound the maximum degree in $G'$ in terms of the number of vertices of $G'$, that is, in terms of $\abs{L}$.
			
Observe that $\abs{L} = n - \abs{S} > n - \frac{n}{C} = n\left(1 - \frac{1}{C}\right)$ by Markov's inequality, and so we get that 
\begin{align*}
    \left(\max_{v \in G'} d_v'\right)^9 \leq \left(C\bar{d}\right)^{10}
    < \left(1 - \frac{1}{C}\right)\frac{n}{2} < \frac{\abs{L}}{2}
\end{align*}
where we, for the first inequality, used that the maximum degree of $G'$ is at most $C\bar{d}$ by construction, and the second inequality follows from our assumption that $\vartheta = (C \bar d)^{10}n^{-1} < \frac{1}{2}\left(1 - \frac{1}{C}\right)$. 
			
Thus $G'$ will satisfies the condition of Proposition \ref{partitionproposition}, and we get a partition $\{A,B,R\}$ of $V(G')$, and thus a partition $\{A,B,R,H\}$ of the vertices of $G$. Since we cut off the vertices of the highest degree, we get the following guarantees on this partition: 
\begin{enumerate}
    \item $\abs{A} = \abs{B}$,
    \item $\abs{R} < (C\bar{d})^9$,
    \item $e(A,B) \leq  \frac{1}{2} \sum_{v \in A \cup B} d_v' \left(\frac{1}{2} - \frac{c}{\sqrt{d_v'}}\right)$,
    \item $\abs{\Vol(A) - \Vol(B)} \leq (C\bar{d})^9(C\bar{d} - \delta) \leq (C\bar{d})^{10}$,
    \item $\max_{v \in R} d_v \leq 2\bar{d}$.
\end{enumerate}
			
Our strategy will be to divide the vertices we do not have degree bounds for -- the ones in $H$ and $R$ -- randomly into $A$ and $B$, and use this randomness to control their contribution to the modularity. As before, the positive contribution to the modularity score will come the fact that there are relatively few edges between $A$ and $B$.
			
Let $U = H \cup R$, and let $\{U_A, U_B\}$ be a partition of $U$. We will first perform some general computations, and then just after~\eqref{termtwobound} we specify the random procedure to partition $U$ into $\{U_A, U_B\}$.
By Remark~\ref{remark_twopartscomputation}, we see that
\begin{align}
    q(G, \{A \cup U_A, B \cup U_B\}) &= \frac{1}{2} - \frac{e(A \cup U_A, B \cup U_B)}{m} - \frac{(\Vol(A \cup U_A) - \Vol(B \cup U_B))^2}{8m^2}.\nonumber
\end{align}
Now substituting $m = e(A \cup B) + e(U) + e(A \cup B, U)$ and \[e(A \cup U_A, B \cup U_B)=e(A,B) + e(U_A, B) + e(A, U_B) + e(U_A, U_B)\] we can compute that
\begin{align*}
     \frac{1}{2} - \frac{e(A \cup U_A, B \cup U_B)}{m} &=  
    \frac{e(A\cup B) - 2e(A,B)}{2m} + \frac{e(U) + e(A\cup B, U)}{2m}\\
    &\qquad - \frac{e(U_A,U_B) + e(U_A, B) + e(A, U_B)}{m}.
\end{align*}
Then, we observe that 
\begin{align*}
    \left(\Vol(A \!\cup \!U_A) - \Vol(B\! \cup\! U_B)\right)^2 
    =   & \; (\Vol(A) - \Vol(B))^2 + 2(\Vol(A)-\Vol(B))(\Vol(U_A)- \Vol(U_B))\\
        &  + (\Vol(U_A) - \Vol(U_B))^2
\end{align*}
and so, taking this and our previous computations we have the following expression for the modularity score
\begin{align}
    q(G, \{A \cup U_A, B \cup U_B\}) &= \frac{e(A \cup B) - 2e(A,B)}{2m} \label{termone}\\
    &- \frac{(\Vol(A) - \Vol(B))^2}{8m^2}\label{termtwo}\\
    &+ \frac{e(U) + e(A\cup B, U)}{2m} - \frac{e(U_A,U_B) + e(U_A, B) + e(A, U_B)}{m}\label{termthree}\\
    &- \frac{(\Vol(A)-\Vol(B))(\Vol(U_A) - \Vol(U_B))}{4m^2}\label{termfour}\\
    &- \frac{(\Vol(U_A) - \Vol(U_B))^2}{8m^2}\label{termfive}
\end{align}
and thus we have five different terms which we will consider in turn.
Firstly, for \eqref{termone}, we use that
    $$e(A,B) \leq \frac{1}{2}\sum_{v \in A \cup B} d_v' \left(\frac{1}{2} - \frac{c}{\sqrt{d_v'}}\right) = \frac{1}{2}e(A \cup B) - \frac{c}{2}\sum_{v\in A \cup B} \sqrt{d_v'}$$
and so \eqref{termone} is bounded below by
\begin{equation}\label{termonebound}
\frac{c}{2m}\sum_{v\in A \cup B} \sqrt{d_v'}.\end{equation}
For \eqref{termtwo} we use our bound on $\abs{\Vol(A) - \Vol(B)}$ to see that
\begin{equation}\label{termtwobound}\frac{\left(\Vol(A) - \Vol(B)\right)^2}{8m^2} \leq \frac{(C\bar{d})^{20}}{8m^2}.\end{equation}
Now, upon reaching the terms that involve $U_A$ and $U_B$, we specify how these sets are chosen.\\

\begin{figure}
    \includegraphics[width=\textwidth]{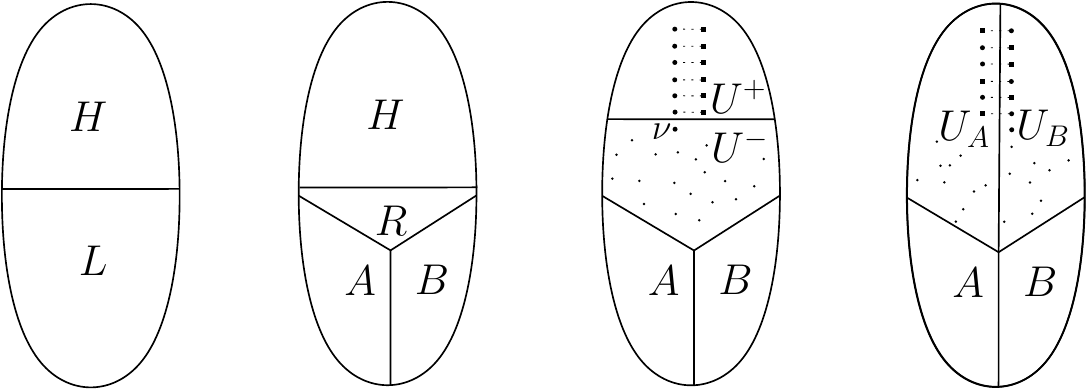}
    \centering
    \caption{\footnotesize 
    { An illustration of the procedure for choosing the partition $\{ A \cup U_A, B \cup U_B\}$. (i) Let $L$ be the set of (low degree) vertices, those of degree at most $C\bar{d}$ and $H$ be the (high degree) vertices, $H=[n] \backslash L$. (ii)~Apply the weight-balancing result to $L$ to get a partition into $A, B$ and a small (remainder) set $R$. (iii) Let $U=H\cup R$, and further divide this into (very high degree) $U^+$, an even number of vertices of degree at least~$n^{1/2}$ and let $U^-=U\backslash U^{+}$ and note $U^-$ may contain one vertex $\nu$ with degree at least $n^{1/2}$. In $U^+$, match the highest degree vertex to the second highest degree vertex, and the third highest to the fourth highest etc. (iv) For each pair in $U^+$ randomly place one endpoint in $U_A$ and the other in $U_B$; for each vertex in~$U^-$ randomly place it in $U_A$ or $U_B$.}}    \label{fig:randombipartition}
\end{figure}

\needspace{4\baselineskip}
\noindent\emph{Random procedure for choosing $U_A$ and $U_B$ (see also Figure~\ref{fig:randombipartition}).}\\
Let $U^+ \subseteq H \subseteq U$ be the set of vertices of degree at least $n^{1/2}$ in $G$, with potentially one vertex less to make $\abs{U^+}$ even. 

Firstly, pick a perfect matching $\mathcal{M}$ on $U^+$, matching the highest-degree vertex to the second-highest-degree, the third highest to the fourth, and so on. 
Secondly, for each edge $xy \in \mathcal{M}$, choose uniformly at random whether to put $x$ in $U_A$ and $y$ in $U_B$, or vice versa. Thirdly, the vertices of $U^-= U \setminus U^+$ get placed into $U_A$ or~$U_B$ independently at random with probability $\frac 12.$ We emphasise that $U^-$ might contain one vertex $\nu$ with $d_{\nu} \geq n^{1/2}$, and the remaining vertices have degree at most $n^{1/2}$.

Note that, since $d_v \geq n^{1/2}$ for $v\in U^+$, we have $n^{1/2}|U^+| \leq \vol(U^+)$. Moreover, $U^+\subseteq H$ and so $\vol(U^+)\leq m$ and thus $\abs{U^+}\leq m/n^{1/2}$. Hence
$$|E(G)\cap \mathcal{M}| \leq \frac{|U^+|}{2} \leq { \frac{m}{2n^{1/2}}}.$$
Having defined our choice of $U_A$ and $U_B$, we can compute the expectation of the remaining terms of the modularity score, i.e.\ \eqref{termthree}-\eqref{termfive}. Starting with the random part of \eqref{termthree}, we get
\begin{align*}
    \E{e(U_A,B)} &= \mathbb{E}\Bigg[ \sum_{%\substack{
                  %x \in U\\ y \in B\\ xy \in E(G) }
                  ub \in E(U,B)}  \ind{u \in U_A} \Bigg]   = \frac{1}{2}e(U,B) 
\end{align*}
and likewise $\E{e(A, U_B)} = \frac{1}{2}e(A,U)$. Moreover, since for $x, y \in U$ that are not matched by~$\mathcal{M}$, their assignment to parts is independent, while for $xy \in M'$ { the endpoints $x$ and $y$ are deterministically assigned} to different parts, we have
\begin{align*}
    \E{e(U_A,U_B)} 
    &= { \tfrac{1}{2}| E(U) \backslash \mathcal{M}|  + |E(U) \cap \mathcal{M}| =  \tfrac{1}{2}e(U) } + \tfrac{1}{2}|E(U) \cap \mathcal{M}|  
    \leq \frac{1}{2}e(U) + { \frac{m}{4n^{1/2}}}.
    \end{align*}
In total, the expectation of the random part of \eqref{termthree} is bounded below by
    $$\frac{e(U,B) + e(A,U) + e(U) }{2m} + { \frac {1}{4n^{1/2}}}$$
which then cancels nearly exactly with the deterministic first half, leaving us with a lower bound on the expectation of \eqref{termthree} of form
\begin{equation}\label{termthreebound} {  %- \frac{m}{8km} \geq 
- \frac{1}{4n^{1/2}} } . \end{equation}
For \eqref{termfour}, we compute that
\begin{align*}
    \E{\Vol(U_A) - \Vol(U_B)} &= \E{\sum_{v\in U_A} d_v - \sum_{v\in U_B} d_v}
    = \E{\sum_{v\in U} d_v(\ind{v \in U_A} - \ind{v \in U_B})}=0.
\end{align*}
and thus \eqref{termfour} term has expectation zero.
{Finally, for \eqref{termfive}, writing $U_A^+=U_A \cap U^+$ and $U_A^-=U_A \cap U^-$, and defining $U_B^+, U_B^-$ similarly, we first note} %(since $(x+y)^2\leq 2x^2+2y^2$)
\begin{align}
\left(\Vol(U_A) - \Vol(U_B) \right)^2  
    & = \left(\Vol (U_A^+) - \Vol(U_B^+) + \Vol(U_A^-) - \Vol(U_B^-)\right)^2 \nonumber \\
    & \leq 2 \left(\Vol (U_A^+) - \Vol(U_B^+) \right)^2 + 2 \left(\Vol (U_A^-) - \Vol(U_B^-) \right)^2. \label{eq:vol-two-terms} 
\end{align}
For the first term of \eqref{eq:vol-two-terms}, Lemma~\ref{lem.obscure} (the load balancing lemma) gives the deterministic bound  $\large|\Vol (U_A^+) - \Vol(U_B^+)\large| \leq \Delta$, where $\Delta$ is the maximum degree of $G$, since we may take $L=1$ in the application of Lemma~\ref{lem.obscure}. The contribution of $U^-$ is 
 \begin{align*}
    &\E{(\Vol(U_A^-) - \Vol(U_B^-))^2} = \mathbb{E} \Bigg[ \bigg(\sum_{v \in U^-} d_v^2 \Big(\ind{v \in U_A } - \ind{v \in U_B} \Big) \bigg)^2\Bigg] \\
    & \quad = \sum_{v \in U^-} d_v^2 \leq d_{\nu}^2 + n^{1/2} \sum_{v \in U^-} d_v  \leq \Delta^2 + n^{1/2}m;
\end{align*}
    recalling that $d_{\nu}$ comes from a potentially unmatched high-degree vertex $\nu$. 
    Thus by~\eqref{eq:vol-two-terms} and using $m \geq n$, we conclude the expected value of~\eqref{termfive} is at least
    \begin{align}\label{termfivebound}
        -\frac{ \E{\left(\Vol(U_A) - \Vol(U_B) \right)^2  }}{8m^2} \geq -\frac{\Delta^2}{m^2} - \frac{ n^{1/2} }{8m} \geq -\frac{\Delta^2}{m^2} - \frac{1}{8n^{1/2} }.
    \end{align}
We may take an instance of the random partition, $\tilde{U}_A$ and $\tilde{U}_B$ say, for which the modularity score of $q(G, A \cup \tilde{U}_A, B \cup \tilde{U}_B)$ is bounded below by our bound on the expectation of the modularity score of the random partition. %$\mathbb{E}[q(G, A \cup U_A, B \cup U_B)]$. 
Gathering our calculations - terms~\eqref{termone} and \eqref{termtwo} are bounded in~\eqref{termonebound} and~\eqref{termtwobound}, and the expectation of terms~\eqref{termthree}-\eqref{termfive} are bounded by line~\eqref{termthreebound}, 0  and line~\eqref{termfivebound} respectively. Thus,
\begin{align}\label{withoutcutoffs_unsimplifiederrorterms}
    q(G, \{A \cup \tilde{U}_A, B \cup \tilde{U}_B\}) &\geq \frac{c}{2m}\sum_{v \in A \cup B} \sqrt{d_v'} 
    - \frac{(C\bar{d})^{20}}{8m^2} \textcolor{violet}{ 
    - { \frac{3}{ 8n^{1/2} } }
    - \frac{\Delta^2}{m^2}}.
\end{align}
It remains to simplify the lower bounds in~\eqref{withoutcutoffs_unsimplifiederrorterms}. We have
\begin{align}\label{withoutcutoffs_claimapplied}
    \frac{c}{2m}\sum_{v \in A \cup B} \sqrt{d_v'} %&= \frac{c}{2m}\sum_{v \in A \cup B} \frac{d_v'}{\sqrt{d_v'}}\nonumber\\
    &\geq \frac{c}{2m}\sum_{v \in A \cup B} \frac{d_v'}{\max_{w \in A\cup B}\sqrt{d_w'}}\nonumber\\
    &= \frac{c\Vol_{G'}(A \cup B)}{2m\left(\max_{v \in A\cup B} \sqrt{d'_v}\right)}\nonumber\\
    &\geq \frac{c\Vol_{G'}(A \cup B)}{2m\sqrt{C \bar{d}}}.
\end{align}
We now consider $\Vol_{G'}(A \cup B)$. Recall that we defined $G' = G[L]$, and $L = A\cup B \cup R$. We compute that
\begin{align*}
    \Vol_{G'}(A \cup B) &= \Vol_{G'}(G') - \Vol_{G'}(R)\\
    &= \Vol_{G}(L) - e_G(H,L) - \Vol_{G'}(R)\\
    &\geq \Vol_{G}(L) - \Vol_G(H) - \Vol_G(R)
\end{align*}
and so since $\Vol_G(L) = (1+\gamma)m$ and $\Vol_G(H)=(1-\gamma)m$ by assumption, we get that $\Vol_{G'}(A \cup B) \geq 2\gamma m - \Vol_G(R)$. Moreover,  $\Vol_G(R)\leq |R|\max_{v\in R} d_v \leq (C\bar d)^9\cdot 2 \bar d = 4m(C\bar d)^9 / n$ by~Proposition \ref{partitionproposition}. 
Substituting this into \eqref{withoutcutoffs_claimapplied}, we get that
$$\frac{c}{2m}\sum_{v \in A \cup B} \sqrt{d_v'} \geq \frac{c\Vol_{G'}(A \cup B)}{2m\sqrt{C \bar{d}}} \geq \frac{c\gamma}{\sqrt{C\bar{d}}} - \frac{2c  (C \bar d)^9}{n\sqrt{C\bar{d}}}.$$
Hence, \eqref{withoutcutoffs_unsimplifiederrorterms} implies that
\begin{align}\label{boundsinserted}
    q(G, \{A \cup \tilde{U}_A, B \cup \tilde{U}_B\}) & \geq 
    		 \frac{c\gamma}{\sqrt{C\bar{d}}} 
        - \frac{2c  (C \bar d)^9}{n\sqrt{C\bar{d}}} 
        - \frac{(C\bar{d})^{20}}{8m^2}
  		  - \frac{3}{8n^{1/2}} - \frac{\Delta^2}{m^2},
\end{align}
and it remains to express two of the error terms in terms of $\vartheta = (C\bar d)^{10}n^{-1}$. Namely, 
$${\frac{2c  (C \bar d)^9}{n\sqrt{C\bar{d}}} } = \frac{2c \vartheta}{(C \bar d)^{3/2}} \text{\quad and} $$
$$\frac{(C\bar{d})^{20}}{8m^2} = \frac{\vartheta^2 n^2}{2n^2 \bar d^2} = \frac{ \vartheta^2}{2 \bar d^2}.$$
Gathering terms and simplifying, we get the final form of our theorem, stating that
\begin{align*}
    \q(G) \geq q(G, \{A \cup \tilde{U}_A, B \cup \tilde{U}_B\}) &\geq \frac{c }{\sqrt{C \bar{d}}} \left(\gamma - \frac{2 \vartheta}{C \bar d} - \right) - \frac{\vartheta^2}{2 \bar d^2} - \frac{3}{8n^{1/2} }-\frac{\Delta^2}{m^2}.
\end{align*}
as desired. Recall that $c >0.26$. It follows that
if $\frac{\Delta}{m} = o(1)$ and $\vartheta=o(1)$, then
 $$\q(G)\geq \frac{0.26 \gamma}{\sqrt{C \bar{d}}}(1-o(1)).$$\end{proof}
The following proposition may be useful to get better constants in some situations, mainly because  we do not lose the $1/ \sqrt{C}$ in the main term.
	\begin{prop}
		Let $G$ be an $n$-vertex graph with average degree $\bar d$, and let $L = \{v \in [n]: d_v \geq C \bar d \}$ for some constant $C \geq 2$. Let $(d_v')_{v \in L}$ be the degree sequence of $G[L]$, and $\vartheta:=(C \bar d)^{10}n^{-1}<1$. Then
        \begin{align}\label{eq:prop-removed-degs}
            \q(G) &\geq \frac{c}{2 n \bar d}\sum_{v \in L} \sqrt{d_v'} -\frac{\vartheta^2}{2 \bar d^2}
            - \frac{3}{8n^{1/2} }
            - \frac{\Delta^2}{4(n \bar d)^2}.
        \end{align}
        	\end{prop}
         \begin{proof}
            We follow the proof of Theorem~\ref{t:main} (with the same notation) down to~\eqref{withoutcutoffs_unsimplifiederrorterms}. Recalling that $L = A \cup B \cup R$, we have
            \begin{align}\label{withoutcutoffs_unsimplifiederrortermsb}
                 q(G, \{A \cup U_A, B \cup U_B\}) &\geq \frac{c}{n \bar d}\sum_{v \in A \cup B} \sqrt{d_v'} 
    					- \frac{(C\bar{d})^{20}}{2(n \bar d)^2} { 
    					- \frac{3}{8n^{1/2}}
    					- \frac{\Delta^2}{4(n \bar d)^2}}.
            \end{align}
            It remains to compare $\sum_{v \in A \cup B }\sqrt{d_v'}$ with $\sum_{v \in A \cup B \cup R} \sqrt{d_v'}$. To this end, note that $|R| \leq (C \bar d)^9 \leq \frac{n}{4} \leq \frac{|A \cup B \cup R|}{2}$, and that $d_w' \leq d_v'$ for all $w \in R$ and $v \in A \cup B$. Therefore,
            $$\sum_{v \in A \cup B }\sqrt{d_v'} \geq \sum_{v \in  R} \sqrt{d_v'},$$
            so            
            $\sum_{v \in A \cup B }\sqrt{d_v'} \geq \frac 12 \sum_{v \in  A \cup B \cup R} \sqrt{d_v'}$, which together with~\eqref{withoutcutoffs_unsimplifiederrortermsb} yields~\eqref{eq:prop-removed-degs}.
         \end{proof}

\section{Lower bounds for power-law graphs}\label{sec:power_law}
We will now apply our main result to deduce Theorem~\ref{t:powerlaw}, as well as a more general lower bound in terms of the moments of the degree sequence. Notice that although Theorem~\ref{t:powerlaw} is stated for constant $\bar d$, the bound actually holds for mildly increasing $\bar d$, up to $\bar d = n^{o(1)}$.
\begin{theorem}[restatement of Theorem~\ref{t:powerlaw}]\label{t:powerlaw_restate}
Let $G$ be a graph with degree sequence $\dv = (d_i)_{i \in [n]}$, with average degree $\bar d$, satisfying 
\begin{equation}  \label{eq:powerlaw-hyp_restate}
    \frac 1n |\{i: d_i \geq k \}| \leq A \bar d^{\tau -1} k^{1-\tau}
\end{equation}
for all $i$, with constants $\tau >2$ and $A >0$. For $b = 0.1 \left( \frac{(\tau-2)}{8A} \right)^{\frac {1}{2(\tau -2)}}$ and sufficiently large $n$,
$$\q(G) \geq b \bar d^{-1/2}.$$
\end{theorem}
\begin{proof} [Proof of Theorem~\ref{t:powerlaw}] 
For convenience, we may assume that $\tau \leq 3$; indeed, if a sequence satisfies~\eqref{eq:powerlaw-hyp} with some $\tau'$, then it also satisfies it with some smaller value $\tau < \tau'$. 
 
To verify the hypothesis of Proposition of~\ref{prop:main}, let $s_j = |\{i: d_i \geq j \}|$ and note that $s_{j}- s_{j+1} = \abs{\{i: d_i = j \}}$. We have
\begin{eqnarray*}
    \sum_{i \in [n]: d_i \geq k} d_i &= \sum_{j \geq k} j^{} (s_j - s_{j+1}) = \sum_{j \geq k} j s_j - \sum_{j \geq k+1} (j-1) s_j = k s_k + \sum_{j \geq k+1} s_j\\
    & \leq A \bar d^{\tau -1} k^{2-\tau}n + \sum_{j \geq k+1} A\bar d^{\tau -1} j^{1-\tau} n,
\end{eqnarray*}
where we second equality follows by changing the summation variable in the second sum, and the inequality uses the hypothesis on $\dv$.
		
Since 
$$\sum_{j \geq k+1} j^{1-\tau} \leq \int_{k}^{\infty} x^{1-\tau} dx = \frac{1}{\tau-2} k^{2-\tau},$$
we have
$$ \sum_{i \in [n]: d_i \geq k} d_i \leq \left( 1 + \frac{1}{\tau -2} \right) A \bar d^{\tau -1} k^{2-\tau} n.$$
Inserting $k = \left(4 A \cdot \frac{\tau -1}{\tau-2}  \right)^{1/(\tau -2)} \bar d$ and dividing by $n \bar d / 2$, we obtain
$$ \frac{2}{n \bar d} \sum_{i \in [n]: d_i \geq k} d_i  \leq \frac{2A (\tau -1)}{ \tau -2} \cdot \left(4 A \cdot \frac{\tau -1}{\tau-2}  \right)^{-1}  = \frac{1}{2}.$$
Hence
$$\sum_{i \in [n]: d_i < k} d_i  \geq n \bar d  - \frac{n \bar d}{4} = \frac{n \bar d}{2} \left(1 + \frac{1}{2}\right),$$
and the hypothesis of our proposition is satisfied with $\gamma = \tfrac{1}{2}$ and $C =  \left(4 A \cdot \frac{\tau -1}{\tau-2}  \right)^{\frac{1}{\tau -2}} \leq \left( \frac {8A}{\tau-2}\right)^{\frac{1}{\tau -2}}$. Proposition~\ref{prop:main} then implies that
$$\q(G) \geq 0.26 \gamma  \left( \frac{8A}{\tau-2}  \right)^{\frac{-1}{2(\tau -2)}}  \bar d ^{-1/2} -O(\vartheta) - \frac{4 \Delta(G)^2}{n^2 \bar d^2}.$$
Now, recall $\vartheta = (C\bar{d})^{10}/n$ for the value $C$ earlier and thus $\vartheta=O(n^{-1})$. For the other error term note~\eqref{eq:powerlaw-hyp_restate} implies that $\Delta(G) \leq (An)^{-\tfrac{1}{1-\tau}} \bar{d}$. 
It follows that, for sufficiently large~$n$, 
		$$\q(G) \geq 0.1 \left(\frac{8A}{\tau-2}  \right)^{\frac{-1}{2(\tau -2)}}  \bar d ^{-1/2}.$$
	\end{proof}
	Let us also point out a more general statement, which controls modularity in terms of moments of the degree sequence. The moments are one way to capture an assumption that the degree distribution is still `reasonably smooth'. Note that in the statement below, $\kappa$ can be an arbitrarily small positive real number, to circumvent the fact that for some graph classes occurring in practice, not even the second moment of the degree sequence is bounded. This statement formally implies Theorem~\ref{t:powerlaw}, but verifying this implication is as difficult as proving Theorem~\ref{t:powerlaw} directly. 
	\begin{prop}
		\label{prop:moments}
		Let $G$ be a graph with degree sequence $\dv = (d_1, \ldots, d_n)$ whose mean is $\bar d = O(1)$. Suppose  for some $\kappa >0$ and $B>0$,
		\begin{equation}
			\label{eq:beta-moment}
			\sum_{v \in [n]} d_v^{1 + \kappa} \leq B n \bar d ^{1+\kappa}
		\end{equation}
		There is a constant $c'$ such that for sufficiently large $n$, $\q(G) \geq c' \bar d ^{-1/2}$.
	\end{prop}
	\begin{proof}
		Let $L$ be the set of vertices of degree at most $(4B)^{1/\kappa} \bar d$ and denotes its complement by $H=L^c$. We claim that then $\Vol(H) \leq \bar d n  / 4$.
		For,
		$$B \bar d ^{1+\kappa} n \geq \sum_v d_v^{1+ \kappa} \geq \sum_{v \notin T} d_v^{1 + \kappa} \geq \left(\min_{v \in T^c} d_v \right)^\kappa \cdot \Vol(H).$$
		Noting that $d_v^\kappa \geq 4B \bar d ^\kappa$ for $v \in H$  and rearranging gives
		$$\frac{ n\bar d}{4} \geq \Vol (H),$$
		as required.
		
		Hence, $\Vol(L) \geq \frac{3n\bar d}{4} $, so we may apply Theorem~\ref{t:main} with $C = (4B)^{1/\kappa}$ and $\gamma = \frac 12$. It follows that
		$$\q(G) = \Omega\left(B^{-\frac{1}{2 \kappa}} \bar d ^{-\frac 12}\right),$$
		where $\Omega$ hides an absolute constant.
	\end{proof}
    
    \def\PA{\mathrm{PA}_n^{(m, \delta)}}
    \def\merror{\left(1 + O(m^{-1}) \right)}
	
\subsection{Preferential attachment graphs and related models}\label{sec:pam}
Preferential attachment models (PAM) describe graphs which grow in time, that is, vertices are sequentially added to the graph. Given the graph at time $t$, a vertex with label $t+1$ is added to the graph and attached to older vertices according to a probability distribution according to which it is more likely to attach to high-degree vertices. Thus the degree sequence of such a graph is not specified a priori, but emerges from the attachment rule. The degree sequence of the classical PAM  considered for instance in~\cite{ab99,brst01} typically follows a power-law with the exponent $\tau = 3$. 

In this section, we demonstrate how Theorem~\ref{t:main} can be applied to an entire class of PAMs which \textit{realise} every power-law exponent $\tau$ with $\tau >2$. We will be working with the model presented in~\cite[Section 8.2]%page 258
{vanderhofstad17}, and we follow their notation. In a graph $G$ on the vertex set~$\{v_1, \ldots, v_n\}$ let $D_i(n)$ denote the degree of $v_i$, and let
    $$P_k(n) = \frac 1n \{i \in [n]: D_i(n) = k \}$$
be the proportion of vertices of degree $k$. At time $t$ the graph has vertex set~$\{v_1, \ldots, v_t\}$ and vertex $i$ has degree $D_i(t)$.

The model has parameters $m \in \N$, which governs the average degree, and $-m <\delta < m$. It produces a graph sequence denoted by $\PA$ which, at time $n$, has $n$ vertices and $mn$ edges. The first vertex $v_1$ has $m$ loops. At time $t$, the vertex $v_t$ is added, along with $m$ edges $e_1, \ldots, e_m$ incident to $v_t$. The other endpoint of the edge $e_i$ is a vertex $v_j \in \{v_1, \ldots, v_t \}$ with probability \textit{roughly} proportional to $D_i(t) + \delta$ (that is, an affine function of the current degree of $v_i$). For a full description of the model see~\cite{ross13,vanderhofstad17} from which all the results which we use are taken. 
We remark that the average degree of this graph is $2m$, %was: m, 
    which does conflict with the use of $m$ (for the number of edges) in the previous section.

    For this specific model, Ross~\cite{ross13} showed that the degree sequence follows a power-law with exponent $\tau = 3 + \frac{\delta}{m} > 2$. Such results were first obtained by Bollob\'as and Riordan~\cite{br03}, for the less general model with $\delta = 0 $ and $\tau =3$. Thus the results of the previous section in principle imply that such graphs have high modularity, but to prove a rigorous result, we need to deal with loops and multiple edges in the model, as well as with the fact that the results from~\cite{vanderhofstad17} (and also~\cite{br03,ross13}) do not a priori give sufficient bounds on the number of vertices of degree, say $n^{1/5}$.

    Recall that $P_k(n)$ is the proportion of vertices of degree $k$ in $\PA$. Let $p_k = p_k(m,\delta)$ be the probability mass function defined in~\cite{vanderhofstad17} and in Appendix~\ref{sec:limiting-distn}; $p_k(n)$ will be the \textit{limiting degree distribution} for $\PA$, and for the present we will only use the estimates
    \begin{equation}
        \label{eq:pam-dist}
        p_k =  k^{-3+\frac{\delta}{m}} \left(2 + \frac{\delta}{m} \right)\left(m+\delta \right)^{3 + \frac{\delta}{m}} \left(1 + O(m^{-1}) \right) \leq 2^5 k^{-3+\frac{\delta}{m}}m^{3+\frac{\delta}{m}},
    \end{equation}
    where the second inequality follows from $3 + \frac{\delta}{m} \leq 4$.
    Let $\tau = -3+\frac{\delta}{m}$. We will need the following facts deduced from~\cite{vanderhofstad17}; { the proof is deferred to after the main theorem, see page~\pageref{proof:pam-deg}}.
    \begin{prop}
        \label{prop:pam-deg}
        With high probability, the following holds in $\PA$ with $\delta \in (-m, m)$.
        \begin{enumerate}
            \item\label{pa-degseq}
                For $k \in [n]$, $k \leq n^{1/10}$, and some $\eps_1 >0$,
                \begin{equation} \label{eq:deg-pam}
                   P_k(n) = p_k(1+O(n^{-\eps_1})).
                \end{equation}
            \item \label{pa-degsum}
                For $A \leq n^{1/10} \log^{-1}n$,
                    $\displaystyle \sum_{k \geq Am} kP_k(n) \leq 2m \cdot 32A^{-\tau+2}/(\tau-2) $
            \item \label{pa-sm}
                $\displaystyle \sum_{k \in [n]} k^2 P_k(t) \leq n^{1-\eps_2}$ for some $\eps_2 >0$.
            \item \label{pa-multiples} 
            The number of loops in $\PA$ is $O(\log^2 n)$, and the number of multiple edges is at most $n^{1-\eps_3}$ for some $\eps_3>0$.
        \end{enumerate}
    \end{prop}

Now we can prove the desired bound. As mentioned the case $\delta=0$ was proven in~\cite{opr17}.
\begin{theorem}\label{thm:pa}
    Let $\tilde{G}$ be an $n$-vertex graph obtained from $G\sim\PA$ after removing loops and multiple edges from $G$, and let $\delta \in (-m, m)$. There is a constant $c'$ such that whp $\tilde{G}$ has average degree $2m(1-o(1))$, and
        $$\q(\tilde{G}) \geq c' m^{-1/2}.$$
\end{theorem}
\begin{proof}
    Assume that $\PA$ satisfies the claims in Proposition~\ref{prop:pam-deg}, which occurs with high probability. Recall that $\tau = 3 + \frac{\delta}{m} >2$. Recall $D_i(n)$ is the degree of vertex $i$ in $G$. Let $d_{\tilde{G}}(v_i)$ denote the degree of $v_i$ in $\tilde{G}$, and clearly $d_{\tilde{G}}(v_i) \leq d_G(v_i) = D_i(n)$.
    
    Let $A$ be a sufficiently large constant such that $\frac{32 A^{-\tau+2}}{\tau-2} < \frac 18$, let $H$ be the set of vertices with degree at least $Am$, and denote its complement by $L=H^c$. By item~\ref{pa-degsum},
        $$\Vol(H) = \sum_{v \in H} d_{\tilde{G}}(v) \leq n\sum_{k \geq Am} k P_k(n) \leq 2mn \cdot \frac{32A^{-\tau+2}}{\tau-2} \leq \frac 18 \cdot 2mn.$$
    By item~\ref{pa-multiples}, $e(\tilde{G}) = mn(1-o(1))$, so
        $$\Vol(L) \geq \frac 78  \cdot 2mn (1-o(1)) = \frac 74 e(G)(1+o(1)) \geq \frac 32 e(\tilde{G}).$$
    Hence we may apply Theorem~\ref{t:main} with $C = A$ and $\gamma = \frac 12$ to deduce that $\q(G) \geq c' m^{-1/2}$, as required.
\end{proof}
    \begin{remark}
        For the classical preferential attachment model, we have $\delta = 0$ and $\tau = 3$, so Theorem~\ref{t:main} can be applied with $A = 2^{8}$ to obtain an explicit value for $c'$.
    \end{remark}
    Before proving Proposition~\ref{prop:pam-deg}, we need some properties of the sequence $p_k$; the formal definition of $p_k$ and the proof of the following lemma can be found in the Appendix.
    \begin{lemma}
        \label{l:neg-bin}
        Let $m$ be a positive integer, $\delta \in (-m, m)$ and $\tau = -3 + \frac{\delta}{m}$. The sequence $p_k=p_k(m, \delta)$ satisfies $\sum_{k =m}^\infty p_k = 2m$. Moreover, there is a constant  $b_{m, \delta}$ such that
        \begin{align}
            \label{eq:pk-first}
            \sum_{k=Cm}^\infty k p_k &\leq \frac{2^5}{\tau -2}C^{2 - \tau}m \text{\quad and}\\
            \label{eq:pk-second}
            \sum_{k=m}^{M} k^2 p_k &\leq b_{m, \delta} \max \{ M^{3-\tau}, \log M\}.
        \end{align}

    \end{lemma}
We can now prove Proposition~\ref{prop:pam-deg}.
\begin{proof}[Proof of Proposition~\ref{prop:pam-deg}]\label{proof:pam-deg}
    Theorem 8.3 in~\cite{vanderhofstad17} states that whp, for all $k$,
    $$|P_k(n)-p_k| \leq \frac{\log n}{\sqrt{n}}.$$
    It follows from \eqref{eq:pam-dist} that for $k \leq n^{1/10}$ and $\tau < 4$, we have $p_k \geq n^{-4/10}$. These two facts together imply~\ref{pa-degseq} holds (for any fixed $\eps_1<1/10$). 

    For item~\ref{pa-degsum}, notice that Lemma~\ref{l:neg-bin} implies that
        $$\sum_{k=m}^{Am} k p_k \geq 2m \left(1- \frac{2^4}{\tau-2}A^{2-\tau} \right).$$
Item~\ref{pa-degsum} will follow from the `complementary inequality'
        \begin{equation}\label{eq:complementary}
        \sum_{k=m}^{Am} kP_k(n) \geq 2m\left(1- \frac{2^5}{\tau-2}A^{2-\tau} \right),\end{equation}
    since $\sum_{k \geq m} k P_k(n) =2m$ deterministically. Now, notice that Lemma~\ref{l:neg-bin} implies that
        $$\sum_{k=m}^{Am} k p_k \geq 2m \left(1- \frac{2^4}{\tau-2}A^{2-\tau} \right).$$
    This estimate and~\ref{pa-degseq} yield~\eqref{eq:complementary}.

        For~\ref{pa-sm}, we split into two ranges. For $k \leq n^{1/11}$, by item (i), we have
        $$\sum_{k \leq n^{1/11}} k^2 P_k(n) \leq  \sum_{k \leq n^{1/11}}2 k^2 p_k.$$
        Thus by~\eqref{eq:pk-second} we have
        $$\sum_{k \leq n^{1/11}} k^2 P_k(n) \leq b_{m, \delta} \max \{ n^{\frac{1}{11}(3-\tau)} , \log n\} \leq b_{m, \delta} n^{1/11}.$$
        Now, by~\ref{pa-degsum}, the sum of all vertex degrees in $\PA$ which are higher than $n^{1/11}$ is at most $C_{m,\delta}' n^{1+(2-\tau)/11} \leq n^{1-\eps}$ for some $\eps > \frac{\tau-2}{11}>0$. Hence, by convexity, the sum $\sum_{k \geq n^{1/11}}k^2 P_k(n)$ is maximised when there is a single vertex of degree $\ell =\lfloor n^{1-\eps} \rfloor$, so
        $$\sum_{k \geq n^{1/11}}k^2 P_k(n) \leq n^{2-2\eps}\cdot \frac{1}{n} \leq n^{1-2\eps}.$$
        Summing the two results gives the required bound.

        For~\ref{pa-multiples}, we let $\cEvent$ be the event that $\PA$ satisfies~\ref{pa-degseq}-\ref{pa-sm}, and we may condition on~$\cEvent$ as it occurs with high probability. Recall that $D_i(t)$ denotes the degree of the vertex~$v_i$ in $\mathrm{PA}_t^{(m, \delta)}$ (i.e.,~after $t$ vertices are added to the preferential-attachment graph). 
        For the purposes of the present proof, it suffices to use crude upper bounds on the attachment probabilities in $\PA$; moreover, we will only use an upper bound $D_i(t)\leq D_i(n)$ for $t \leq n$. For the exact probabilities, see~\cite[page 258]{vanderhofstad17}.

        The first vertex $v_1$ has $m$ loops. When adding the vertex $v_{t+1}$, $m$ edges are attached to $v_{t+1}$, and each of them is a loop with probability at most $$\frac{2(m-1)}{mt},$$
        where the numerator $2m$ corresponds to the worst-case scenario where $v_{t+1}$ already has $m-1$ loops attached to it. Summing over the $m$ edges attached to $v_{t+1}$ (for $t \geq 1$) and over all $t$, the expected number of loops is at most
        $$m + \sum_{t =1 }^n \frac{m}{t} \leq 2m \log n.$$
        So by Markov's inequality, the number of loops in $\PA$ is at most $\log^2 n$ with high probability.

        To control multiple edges, note that~\ref{pa-sm} implies that conditional on $\cEvent$, $$\sum_{i \in [n]}D_i^2(n)  = n\sum_{k =m}^n k^2P_k(n) \leq n^{2-\eps}.$$
        Let $Z_t$ denote the number of multiple edges $v_i v_{t+1}$ with $i \leq t$. The probability that one of the $m$ edges attached to $v_{t+1}$ is incident to a given vertex $v_i$ (with $i \neq t+1$) is at most $m \cdot \frac{D_i(n)+\delta}{mt(2+\delta)+(1+\delta)} \leq \frac{D_i(n)}{t}$. Hence the probability that $v_i v_{t+1}$ is a multiple edge is at most $\frac{D_i^2(n)}{t^2}$.
        Thus for $t \geq n^{1-\eps/4}$, %the expected number of multiple edges whose second vertex is $v_{t+1}$ is at most
        $$\mathbb{E} [Z_t \mid \cEvent] \leq \frac{1}{t^2}\sum_{i \in [n]}D_i^2(n) \leq n^{-\eps/2}.$$
        Summing over $t$, and using the trivial upper bound $Z_t \leq m$ for $t \leq n^{1- \eps / 4}$, we get that the expected number of multiple edges is at most
        $$\sum_{t = 1}^n  \mathbb{E} [Z_t \mid \cEvent ] \leq  n^{1- \eps / 4} m+ n^{1- \eps/2} \leq 2mn^{1-\eps/4}. $$
        Again, using Markov's Inequality, we have that $\sum_t Z_t \leq n^{1-\frac{\eps}{5}}$ with high probability.
    \end{proof}

	%******************* %*********

	\section{Upper bounds on modularity} \label{sec:upper-random}
\newcommand{\bw}{\mathbf{w}}
In this section, we show that for a large class of sequences $\dv$, \textit{typical} graphs with degree sequence \textit{approximately} $\dv$ actually have modularity $O(\bar d^{-1/2})$, matching the lower bound from Theorem~\ref{theorem_withoutcutoffs} up to a constant factor.

We consider the Chung-Lu model of random graphs. Let $\bw = (w_v)_{v \in [n]}$ where each $w_v>0$  and denote $\bar{w}=n^{-1}\sum_v w_v$ and $w_{\rm min}=\min_v w_v$. We will also assume that for each $v$ we have $w_v^2=o(\bar{w}n)$.  Generate the random graph $G(n, \bw)$ by choosing each edge $uv$ independently with probability (where $u\neq v$ as we do not allow loops) 
\[p_{uv}=\frac{w_u w_v}{\bar{w}n}. \]  
We may see that the expected degree of $v$ in $G(n, \bw)$ is $w_v(1-w_v\bar{w}^{-1}n^{-1})=w_v(1-o(1))$, i.e. approximately 
$w_v$. This is why the Chung-Lu model is often referred to as the random graph with a given \textit{expected} degree sequence. In fact, for a large class of degree sequences, the empirical degree sequence of~$G(n, \bw)$ is \textit{close to}~$\bw$; for details, see Theorems 6.10 and~6.19 in~\cite{vanderhofstad17}. If the degree sequence of $G(n, \bw)$ satisfies the assumptions of Theorem~\ref{sec:without_cutoffs}, then we can deduce that its modularity is $\Omega(\bar w^{-1/2})$. We will now prove an upper bound of the same order of magnitude, assuming that $w_{\min} \geq c \bar w$ for some constant $c$.

Throughout this section, we write \textit{whp} to mean \textit{with high probability}, i.e.~with probability converging to 1 with $n$.
We recall the normalised Laplacian of a graph $G$ is defined to be $\mathcal{L}_G = I - D^{-1/2}AD^{-1/2}$ where $A$ is the adjacency matrix of $G$ and $D$ is the diagonal `degrees matrix' where the $u$-th entry on the diagonal is $d_u$. Let $\bar{\lambda}_G$ be the spectral gap of $\mathcal{L}_G$. A very nice result of Chung, Lu and Vu \cite{chungluvu} is that whp 
\[ \lambda(G(n, \textbf{w})) > 1 - 4\bar{w}^{-1/2} (1 + o(1)) -w_{\rm min}^{-1} \ln^2 n. \]
Now we recall that the modularity of a graph is bounded above by its spectral gap~see for example Lemma~6.1 of~\cite{ERmod}: $\q(G) \leq \bar{\lambda}(G)$. Thus the result of~\cite{chungluvu} immediately gives the following corollary. Also recall that the modularity value is robust to changes in the edge-set, if we may obtain $H$ from $G$ by deleting at most $\eps \cdot e(G)$ edges then $|\q(G)-\q(H)|<2\eps$, by Lemma~5.1 of~\cite{ERmod} (we will use this to obtain Corollary~\ref{cor.ubCOLanka}).  

\begin{cor}\label{cor.ubChungLuVu} Suppose $\bw$ is a degree sequence with $w_{\rm min}=\omega(\ln^2 n)$. Then
\[ \q(G(n, \bw)) \leq  4\bar{w}^{-1/2} (1 + o(1)) .\]
\end{cor}
For a larger class of $\bw$, Coja-Oghlan and Lanka~\cite{coja-oghlan_lanka_2009} show lower bounds on the spectral gap not on the entire graph $G(n, \bw)$ but for an induced subgraph which comprises most of the volume of the graph.

\begin{theorem}[\cite{coja-oghlan_lanka_2009}]\label{thm.COLanka} There exists constants $c_0$ and $w_0$ such that the following holds. If~$\bw$ satisfies $w_0 \leq w_{\rm min} \leq w_{\rm max} \leq n^{0.99}$ then whp  $G$ contains an induced subgraph $H$ with \begin{itemize}
\item[(i)] $\bar{\lambda}_H \geq 1- c_0 w_{\rm min}^{-1/2}$ and
\item[(ii)] $e(H)\geq e(G) - n\exp(-w_{\rm min}/c_0)$. 
\end{itemize}
\end{theorem}

\begin{cor}\label{cor.ubCOLanka} There exists constants $c_0$, $w_0$ such that the following holds. If $\bw$ satisfies  $w_0 \leq w_{\rm min} \leq w_{\rm max} \leq n^{0.99}$, then whp
\[ \q(G(n, \textbf{w})) \leq  c_0 w_{\rm min}^{-1/2} .\]
\end{cor}

\begin{proof}
The corollary follows almost immediately from Theorem~\ref{thm.COLanka}. Since $\E{\vol(G)}=\bar{w} n = \omega(1)$ we get that whp $\Vol(G)=\bar{w}n(1+o(1)) \geq \frac{2}{3}\bar{w}n$. Thus whp $G(n, \bw)$ contains a subgraph $H$ with $\frac{e(H)}{e(G)} \geq 1 - \frac{n}{e(G)} \geq 1 - \frac {3}{\bar w}=1-o(1).$ Hence by the spectral upper bound with high probability,
\[ \q(G(n, \textbf{w})) \leq c_0' \; w_{\min}^{-1/2},
\]
 which implies the result.
\end{proof}

\section{Concluding remarks}
        For a large class of sequences $\mathbf{d}$, we showed that any graph with degree sequence $\mathbf{d}$ has modularity~$\Omega (\bar d^{-1/2})$, improving on the previously known lower bound of order $\bar d^{-1}$. Specifically, this bound applies to graphs with a power-law degree sequence, which includes preferential-attachment graphs (under suitable models).

        However, to our knowledge, the best known upper bound on the modularity of the preferential-attachment graph is $\frac{15}{16}$~\cite{opr17}. Preferential-attachment graphs are not sampled with an inherent community structure, so one might expect their modularity to decay with the average degree $\bar d$, which is also suggested in~\cite{opr17} where they showed a lower bound of $\Omega(\bar{d}^{-1/2})$. It would be very interesting to prove such an upper bound, and perhaps even a bound of order $O(\bar d^{-1/2})$.

\printbibliography
\appendix
\section{Proof of Lemma~\ref{lem.nearlybisection}}

To show how our weight-balanced bisection leads to a high modularity partition we used Lemma~\ref{lem.nearlybisection}. This lemma gives a lower bound on the modularity score of a partition intro three parts: two parts with near equal volume and a remainder part. Here we provide the (short) proof of this lemma, which we repeat below for convenience.  

\begin{lemma}[restatement of Lemma~\ref{lem.nearlybisection}]\label{lem.nearlybisectionagain} Let $G$ be a graph and $\cA=\{A,B,R\}$ a vertex partition of $G$ with $\vol(R)\leq \vol(G)/3$. Then
    \begin{eqnarray*}
        q_\cA(G) &\!\geq \!& 
                     \frac{1}{2}  - \frac{e(A,B)}{e(G)} - \frac{(\vol(A)-\vol(B))^2}{2\vol(G)^2}
    \end{eqnarray*}
\end{lemma}

\begin{proof}[Proof of Lemma~\ref{lem.nearlybisection}]
For vertex set $R$, we write $\partial(R)$ for the number of edges with exactly one endpoint in $R$. Thus the edge contribution for $\cA$ on $G$ is
\[ q^E_\cA(G) = 1 - \tfrac{1}{m}\big(e(A,B) + \partial(R) \big) =  \tfrac{1}{2} + \tfrac{1}{m}\big(\tfrac{1}{2}e(G)-e(A,B) \big) - \tfrac{1}{m}\partial(R) .\]

For the degree tax, roughly speaking, $\vol(R)$ is negligibly small and parts $A$ and $B$ are of similar volume, i.e.\ $\vol(A)\approx \vol(B)\approx \vol(G)/2$. Thus the degree tax is near what it would be if we had two parts of exactly equal volume (which would be $(1/2)^2+(1/2)^2=1/2$).

Define $t$, a measure of the near-ness of the volumes of $A$ and $B$, by
\[\vol(A)=\vol(B)+ t\cdot \vol(G)\]
and $r$, the scaled size of the remainder, by 
\[\vol(R)=r\cdot \vol(G).\]
Hence \[\vol(A)=\vol(G)-\vol(B)-\vol(R) = \vol(G)-\vol(A)+(t-r)\cdot \vol(G)\]
and thus, (with similar calculations for $\vol(B)$), \[\frac{\vol(A)}{\vol(G)}=\tfrac{1}{2}(1+t)-\tfrac{1}{2}r \;\;\;\;\;\;\;\mbox{and}\;\;\;\;\;\;\;\frac{\vol(B)}{\vol(G)}=\tfrac{1}{2}(1-t)-\tfrac12 {r}\]
Now we may calculate the degree tax
\[q_\cA^D(G)=\big(\tfrac{1}{2}(1+t)-\tfrac12 {r}\big)^2+\big(\tfrac{1}{2}(1-t)-\tfrac12 {r}\big)^2+r^2=\tfrac{1}{2}+\tfrac12 {t^2}-r+\tfrac{3}{2}r^2.\]
To put the bounds together, note $\partial(R)\leq \vol(R)$ and hence $\partial(R)/m = \partial(R)/(2\vol(G)) \leq r/2$. Hence the modularity score is at least

\begin{eqnarray*}
    q^E_\cA(G)-q^D_\cA(G) 
    & \geq & \big(\tfrac{1}{2} + \tfrac{1}{m}\big(\tfrac{1}{2}e(G)-e(A,B) \big) - \tfrac{1}{2}r\big) - \big( \tfrac{1}{2}+\tfrac12 {t^2}-r+\tfrac{3}{2}r^2 \big)  \\
    & = & \tfrac{1}{m}\big(\tfrac{1}{2}e(G)-e(A,B) \big) -\tfrac12 {t^2} + \tfrac{1}{2}r -\tfrac{3}{2}r^2 
\end{eqnarray*}
Since $\vol(R) \leq \vol(G)/3$, i.e.\ $r\leq 1/3$, we have $r/2-3r^2/2= r(1-3r)/2 \geq 0$ which yields the result. 
\end{proof}

\section{The limiting degree distribution of the preferential attachment model}
    \label{sec:limiting-distn}
    In this Section, we prove Lemma~\ref{l:neg-bin}, which summarises some properties of the probability mass function $p_k = p_k(m+ \delta)$. The main difficulty is proving that $\sum_{k \geq m} kp_k = 2m$, for which we use an alternative characterisation of $p_k$ as a distribution of a random variable $X$, which can be found in~\cite{vanderhofstad17}.

    Let $X(p)$ be a random variable with distribution
    $$\Prob{X(p)=k} = \frac{\Gamma(m+\delta+k)}{k!\Gamma(r)}p^{m+\delta}(1-p)^{k},$$
    which is usually called the \textit{negative binomial distribution} with parameters $m+\delta$ and $p$. (For $m+\delta \in \N$, $X$ describes the time of the $r$-th success in a sequence of independent experiments with success probability $p$.) We have $\E{X(p)} = {(m+\delta)(1-p)/p}$. Let $U$ be a uniform random variable in $[0, 1]$; then  $X(U^{1/(2+\delta/m)})$ has the negative binomial distribution with a random parameter $p = U^{1/(2+\delta/m)}$. The function $(p_k)_{k \geq m}$ can be described as
         $$p_k = \Eover{U}{ \Prob{X(U^{1/(2+\delta/m))} =k-m } },$$
         often referred to as a \textit{mixed} distribution.
         For convenience, we define $p_k = 0$ for $k<m$.
         We remark that from this description, it is clear that $\sum_k p_k = 0$.
    \begin{proof}
        [Proof of Lemma~\ref{l:neg-bin}]
        Let ${2+\delta/m}=a >1$. We have, using the definition of $p_k$ and linearity of expectation,
        \begin{align*}
            \sum_{k \geq m} k p_k &= \sum_{\ell \geq 0} (m + \ell) p_{m+\ell} 
                = m + \Eover{U} {\sum_{\ell \geq 0} \ell \  \Prob{X(U^{1/a})=\ell} } \\
                & = m + \int_{0}^1 (m+\delta) \cdot \frac{1- u^{1/a} }{u^{1/a}}  \mathrm d u= m + (m+\delta)\cdot  \frac{1}{a-1}  \\
                & = m + \frac{m+\delta}{\frac 1m (m + \delta)}  = 2m;
        \end{align*}
        where the convergence of the integral follows from $1/a < 1$.

      Let $\tau = 3 + \frac \delta m >2$.  For~\eqref{eq:pk-first}, we use the approximation~\eqref{eq:pam-dist}, which we restate here as
    \begin{equation}
        \label{eq:pam-dist1}
        p_k =  \leq 2^5 k^{-\tau}m^{\tau -1}.
    \end{equation}
     We have
        $$\sum_{k = Cm}^\infty kp_k \leq 2^5 m^{\tau-1} \sum_{k=Cm}^\infty k^{1-\tau} \leq 2^5m^{\tau -1} \cdot \frac{(Cm)^{2-\tau}}{\tau - 2} \leq \frac{2^5 C^{2-\tau}}{\tau-2}m,$$
        as required.

        Similarly, for~\eqref{eq:pk-second}, we can subsume all terms that are constant with respect to $n$ into a constant $b' = b_{m, \delta}$ and obtain
        $$\sum_{k=m}^{M} k^2 p_k \leq b'_{m, \delta} \sum_{k=m}^M k^{2-\tau}. $$
        Since the sum $\sum_k k^{2-\tau}$ may diverge (and does for $\tau \leq 3$), it is dominated by the \textit{top terms}, so
        $$\sum_{k=m}^M k^2 p_k \leq b_{m, \delta} \max \{M^{3-\tau}, \log M\}.$$
    \end{proof}
\end{document}